\NewCommandCopy{\latexref}{\ref}
\NewCommandCopy{\latexpageref}{\pageref}

\documentclass{birkjour}

\RenewDocumentCommand{\ref}{m}{{\rm\latexref{#1}}}
\RenewDocumentCommand{\pageref}{m}{{\upshape\latexpageref{#1}}}
 \newtheorem{thm}{Theorem}[section]
 \newtheorem{cor}[thm]{Corollary}
 \newtheorem{lem}[thm]{Lemma}
 \newtheorem{prop}[thm]{Proposition}
 \theoremstyle{definition}
 
 \theoremstyle{remark}
 \newtheorem{rem}[thm]{Remark}
 
 \numberwithin{equation}{section}

\begin{document}

%
%
%
%
%
%
%
%
%

\title[A hybrid Nehari-Schauder fixed point result for systems]
 {Hybrid Nehari-Schauder type fixed point results and applications}

\author[R. Precup]{Radu Precup$^\ast$}

\address{%
Faculty of Mathematics and Computer Science, and \\
Institute of Advanced Studies in Science and Technology, Babe\c{s}-Bolyai University,\\  Cluj-Napoca, 400084, Romania, and \\
Tiberiu Popoviciu Institute of Numerical Analysis, Romanian
Academy,\\ Cluj-Napoca, 400110, Romania}

\email{r.precup@ictp.acad.ro}

\thanks{$^\ast$Corresponding author. Email: r.precup@ictp.acad.ro}
\author[A. Stan]{Andrei Stan}
\address{Faculty of Mathematics and Computer Science, Babe\c{s}-Bolyai University, \\Cluj-Napoca, 400084, Romania, and \\
Tiberiu Popoviciu Institute of Numerical Analysis, Romanian
Academy,\\ Cluj-Napoca, 400110, Romania}
\email{andrei.stan@ubbcluj.ro}


\date{January 1, 2004}

\begin{abstract}
This paper develops a fixed point version of the well-known Nehari manifold method from  critical point theory. The main result is formulated for systems of operator equations, relying on the fixed point theorems of Schauder and Schaefer. The framework also allows for potential extensions combining our Nehari type approach with other fixed point principles. To demonstrate the applicability of the method, an example involving a system of nonlinear integral equations is provided.


\end{abstract}

\maketitle
\section{Introduction and Preliminaries}

In the context of critical point theory, various methods have been developed to establish conditions such that a given functional possesses critical points.
 Among them, we highlight the Nehari manifold method, which has its origins in the classical works of Nehari \cite{nehari1,nehari2}. A remarkable paper about this method is the one  by Szulkin and Weth \cite{sw}, which provides a clear exposition of how this method should be applied, along with illustrative examples.

Typically, following \cite{sw}, the procedure reads as follows: given a  $C^1$ functional \(E\colon X\to \mathbb{R}\) (usually of energy type), where $X$ is a Banach space, one considers the associated \textit{Nehari manifold} defined by
\begin{equation}\label{nehari}
    \mathcal{N} := \left\{ u\in X\setminus\{0\} \;: \; \langle E'(u), u \rangle = 0 \right\}.
\end{equation}
 The main idea of this method is to show that the infimum of \(E\) over \(\mathcal{N}\) is attained at some point \(u_0 \in \mathcal{N}\), and that this point is  a critical point of \(E\). Such a point is usually referred to as a ground state solution, as it minimizes the functional \(E\) among all nontrivial critical points, since every nontrivial critical point of \(E\) lies on the Nehari manifold \(\mathcal{N}\).

This approach is typically effective under the following assumptions:
\begin{itemize}
    \item[(i)] \(E \in C^1(X, \mathbb{R})\), and for each \(u \in X\setminus \{0\}\), the function  \begin{equation*}
     \varphi\colon \mathbb{R}_+\setminus\{0\}\to \mathbb{R},\quad   \varphi(s) := E(su),
    \end{equation*}  admits a unique critical point \(s_u > 0\);
    
    \item[(ii)] \(\varphi'(s)>0\) for $0<s<s_u$ and \(\varphi'(s)<0\) for $s>s_u$;
    
    \item[(ii)] The mapping $u\mapsto s_u$ is continuous.
\end{itemize}
However, when dealing with problems that lack a variational structure, the method described above becomes ineffective. Since a critical point problem is equivalent to a fixed point problem (for instance, by considering \(u = u - E'(u)\)), the motivation of this paper is to adapt ideas from the Nehari manifold method to the setting of a standard fixed point problem of the form \(T(u) = u\). This approach becomes particularly relevant in situations where there is no underlying functional whose derivative is related to the operator \(T\) (e.g., when the problem can not be equivalently expressed as a critical point problem).

The novelty of this paper lies in the development of an entirely new fixed point method, inspired by the Nehari manifold technique, and formulated in the context of a fixed point problem for a system of two operator equations.

The central idea introduced in this paper is to replace the classical Nehari manifold \(\mathcal{N}\) with another set \(U_b\), defined as the collection of all nonzero points in a given domain for which a certain functional \(\mathcal{F}\) vanishes. Thus, given a fixed point problem \((T_1(u,v), T_2(u,v)) = (u,v)\), we derive conditions such that a fixed point exists in \(K_1 \times D\), where \(K_1\) is a cone in a Banach space and \(D\) is a bounded, closed, and convex set of a possibly different Banach space.
To the best of our knowledge, the idea of adapting the Nehari manifold method to the non-variational setting is entirely new, both for single equations and for systems. In the case of systems, our approach further integrates this idea with Schauder’s fixed point theorem through the use of the fixed point index.

We would like to mention the paper \cite{figueredo}, which, using the Nehari manifold method, provides conditions that ensure the existence of a critical point of a functional within a cone with nonempty interior.
 Although the methods presented in \cite{figueredo} are well suited for many applications, a key distinction from our approach is that the cone \(K_1\) in our setting is not required to have nonempty interior.
Also, in \cite{a,pa}, the authors employed the Nehari manifold technique to obtain localization of critical points in conical annular sets, where the cones involved are closed, as in our case, and no openness condition was imposed.

In the variational case, with an appropriate choice of the functional \(\mathcal{F}\), our set \(U_b\) essentially coincides with the Nehari manifold \(\mathcal{N}\). However, a key advantage of our approach is that it allows for the construction of alternative sets that may possess more favorable geometric properties. As a result, this can lead to different and possibly less restrictive conditions for the existence of a critical point.
Of course, it must be emphasized that if a critical point is obtained via our method, and the set \(U_b\) differs from \(\mathcal{N}\) (in the variational setting), then the resulting critical point generally loses the property of being  a ground state solution.


It is worth noting that the literature contains numerous contributions in which various fixed point techniques are combined to study systems of equations. For instance, Krasnosel’skiĭ’s fixed point theorem has been applied componentwise to systems of two equations (see, e.g., \cite{precup_Krasnoselskii,precup_rodriguez2023}), combined with the method of lower and upper solutions \cite{rodriguez2025}, or used jointly with Schauder’s fixed point theorem (see \cite{jorje_Krasnoselskii_schauder}).

Likewise, the proposed non-variational Nehari-type approach can be also  combined with other fixed point principles, leading to other new hybrid results such as Nehari–Krasnosel’skiĭ, Nehari–Darbo, Nehari–Sadovskiĭ, or Nehari–Mönch, under either invariance conditions or Leray–Schauder-type boundary conditions.
Moreover, the method can be extended to operator systems of higher dimension, by combining the Nehari technique with various fixed point principles—for instance, Avramescu’s theorem. 
Another
direction could aim at exploiting the properties of the fixed point degree in combination with
the Nehari technique.

We conclude this introductory section by recalling  some  useful properties of the fixed point index for compact maps. For more details, we refer the reader to \cite{amann,granas} (see also \cite[Section 20.1]{deimling}).

\begin{prop}
Let \(C\) be a closed and convex subset of a normed linear space \(X\), and let \(U\) be a relatively open subset of \(C\). Additionally, let
\[
T : \bar{U} \to C
\]
be a compact map with no fixed points on the boundary of \(U\). Then, the fixed point index of \(T\) in \(C\) over \(U\), denoted by \(\mathrm{ind}_C(T,U)\), has the following properties:
\begin{itemize}
    
    \item[1.](Existence) If \(\mathrm{ind}_C(T,U) \neq 0\), then there exists \(u \in U\) such that \(u = Tu\).
    
    \item[2.](Homotopy invariance)  If 
    \[
    H : \overline{U}\times [0,1] \to C
    \]
    is a compact mapping such that
\[
H(u,t)\,\neq\, u
\quad \text{for all } u\in \partial U \text{ and } t\in [0,1],
\]
then
\[
\mathrm{ind}_C(H(\cdot,1), U) = \mathrm{ind}_C(H(\cdot,0), U).
\]
    
    \item[3.](Normalization)  If \(T\) is  constant  with \(T(u) = u_0\) for every \(u \in \overline{U}\), then
    \[
    \mathrm{ind}_C(T,U) =
    \begin{cases}
    1, & \text{if } u_0 \in U,\\[3pt]
    0, & \text{if } u_0 \in  C\setminus \bar{U}.
    \end{cases}
    \]
\end{itemize}
\end{prop}

\section{Main results}
\subsection{Nehari-Schauder fixed point theorem}
Let \(\left(X_1, |\cdot|_1\right)\) and \(\left(X_2, |\cdot|_2\right)\) be two Banach spaces, let \(K_1 \subset X_1\) be a non-degenerate cone, i.e., a closed, convex set with \(\lambda K_1 \subset K_1\) for all \(\lambda \in \mathbb{R}_+\),  \(K_1 \setminus \{0\} \neq \emptyset\), and \(D \subset X_2\)  a nonempty, bounded, closed, convex set.

We consider the fixed-point problem
\[
\begin{cases}
    T_1(u, v) = u \\
    T_2(u, v) = v,
\end{cases}
\]
where
\[
T = \bigl(T_1, T_2\bigr)\colon K_1 \times D \to K_1 \times D
\]
is a completely continuous operator. 
We also consider a continuous functional
\[
\mathcal{F}\colon K_1 \times K_1 \to \mathbb{R},
\]
with the property that for  any \(u \in K_1 \setminus \{0\}\) and  \(\lambda > 0\), one has
\begin{equation}\label{proprietate F}
    \mathcal{F}\bigl(\lambda\,u,\,u\bigr) = 0 
\,\,\  \text{ if and only if }\,\, 
\lambda = 1.
\end{equation}
Concerning the functional \(\mathcal{F}\), we impose the following conditions:
\begin{description}
    \item[(h1)] For every \( (u,v) \in \left(K_1 \setminus \{0\}\right)\times D\), there exists a unique number \(s(u,v)>0\) such that
    \[
        \mathcal{F}\Bigl( T_1\bigl( s(u,v)\,u,\, v \bigr),\, s(u,v)\,u \Bigr) = 0.
    \]
\end{description}
\begin{description}
\item[(h2)] There exists $0<m<M<\infty$ such that 
\[
\inf_{S_{K_1}^{1}\times D} s(\cdot, \cdot) > m,
\]
and \begin{equation*}
\sup_{S_{K_1}^{1}\times D}s\left( \cdot ,\cdot \right) <M,
\end{equation*}
where
\[
S_{K_1}^{1} := \bigl\{ u \in K_1 : \, |u|_1 = 1 \bigr\}.
\]
\end{description}
The first  result derived from conditions $\text{(h1)-(h2)}$,  essential for our subsequent analysis, is the continuity of the mapping \[
s \colon \bigl(K_1 \setminus \{0\}\bigr)\times D \to (0,\infty).
\]
\begin{prop}
Under assumptions $\emph{(h1)-(h2)}$, the mapping $s$
is continuous.
\end{prop}
\begin{proof}
First, let us note that for any pair \((u, v) \in \bigl(K_1 \setminus \{0\}\bigr)\times D\) and any \(t > 0\), we have
\begin{equation}\label{prop_s_homogeneity}
   s\bigl(t u, v\bigr) = \frac{1}{t}\, s\bigl(u,v\bigr).
\end{equation}
Indeed, denoting
\[
    \lambda: = s(tu,v)\, t,
\]
we observe that
\[
0= \mathcal{F}\bigl( T_1\bigl( s(tu,v)\, t\,u,\, v \bigr),\, s(tu,v)\, t\,u \bigr)
        = \mathcal{F}\bigl( T_1\bigl( \lambda\, u,\, v \bigr),\, \lambda\, u \bigr).
\]
The uniqueness of \(s(u,v)\), ensured by assumption $\text{(h1)}$, implies that
\[
    s(u,v) = \lambda,
\]
so  relation~\eqref{prop_s_homogeneity} is satisfied.

Let \((u, v) \in \bigl(K_1 \setminus \{0\}\bigr)\times D\), and consider any sequence \((u_k, v_k) \subset \bigl(K_1 \setminus \{0\}\bigr)\times D\) such that
\[
(u_k, v_k) \to (u, v) \quad \text{as } k \to \infty.
\]
Based on \eqref{prop_s_homogeneity}, we have \[
s(u_{k},v_{k}) = \frac{1}{\lvert u_{k}\rvert_1}\, s\left(\frac{u_{k}}{\lvert u_{k}\rvert_1},\,v_{k}\right).
\]
Since by $\text{(h2)}$  the sequence \(s\left(\frac{u_{k}}{\lvert u_{k}\rvert_1},\,v_{k}\right)\) is bounded both away from zero and from above, and  \[
\frac{1}{\lvert u_{k}\rvert_1}\;\to\; \frac{1}{\lvert u\rvert_1} \, \, \text{ as }k\to \infty,
\] it follows that the sequence $s(u_k,v_k)$ is also bounded both from below and above, i.e., there exists $0<c<C<\infty$ such that  $s(u_k,v_k)\in [c,C]$ for all $k\in \mathbb{N}$.  In order to prove that $s(u_k,v_k)\to s(u,v)$ as $k\to \infty$, it suffices to show that any convergent subsequence converges to $s(u,v)$ (see, e.g., \cite[Lemma 1.1]{dret}). Thus, let $s(u_{k_q},v_{k_q})$ be a subsequence convergent to some $s_\ast\in [c,C]$. Then, by the definition of the mapping $s$, we have
\[
0 = \mathcal{F}\left( T_1\left(s(u_{k_q},v_{k_q})u_{k_q},v_{k_q}\right), s\bigl(u_{k_{q}},\,v_{k_{q}}\bigr)u_{k_q}\right).
\]
Now, passing to the limit, one obtains that \begin{equation*}
    \mathcal{F}\left( T_1(s_\ast u,v),s_\ast u\right)=0.
\end{equation*}
From \(\text{(h1)}\), there exists a unique \(s(u, v) > 0\) such that
\[
\bigl(\mathcal{F}\bigl(s(u,v)\,u,\,v\bigr),\,s(u,v)u\bigr)\;=\;0,
\]
so \(s_\ast=s(u,v)\) as desired.  This completes the proof, since the choice of \((u, v)\) and the sequence \((u_k, v_k)\) was arbitrary.
\end{proof}
In the subsequent, 
 we consider the sets \begin{equation*}
     U_b=\bigl\{\left(  u,v\right)\in (K_1\setminus \{0\})\times D \, : \, s(u,v)=1 \bigl\}
 \end{equation*}
 and
\begin{equation*}
U=\bigl\{\left( \lambda u,v\right) :0\leq \lambda <1 \,\, \text{ and }\,\, (u,v)\in U_b\bigl \}.
\end{equation*}
\begin{rem}\label{remarca_echivalenta_cu_nehari}
A simple observation based on assumption $\text{(h1)}$  allows us to establish an equivalent characterization of the set \(U_b\), namely,
 \begin{equation*}
          U_b=\left\{\left(  u,v\right)\in (K_1\setminus \{0\})\times D \, : \,  \mathcal{F}\Bigl( T_1\bigl( u,\, v \bigr),\, u \Bigr) = 0 \right \}.
    \end{equation*}
\end{rem}
\begin{lem}\label{U open}
    The set $U$ is open in the relative topology of $K_{1}\times D.$ Moreover, it is a bounded set.
\end{lem}
Before proving this, we need the following auxiliary result. \begin{prop}\label{lema_conv_zero}
 Let $(u_k,v_k) \in \left( K_1\setminus\{0\}\right)\times D$ be a sequence such that \begin{equation*}
     (u_k,v_k)\to (0,v) \,\, \text{ as }k\to \infty,
 \end{equation*} for some $v\in D$. Then, \begin{equation*}
     s(u_k,v_k)\to \infty \,\,\,  \text{ as } k\to \infty.
 \end{equation*}  
\end{prop}
\begin{proof}
Since $u_k\neq 0$,  by assumption $\text{(h2)}$, one has \begin{equation*}
        s\left( \frac{u_k}{|u_k|_1},v_k\right)\geq m  \,\,\, \text{ for all }k\in \mathbb{N}.
    \end{equation*}
   Using  \eqref{prop_s_homogeneity} we see that \begin{equation*}
        s\left( u_k,v_k\right)=s\left(|u_k|_1 \frac{u_k}{|u_k|_1},v_k\right)=\frac{1}{|u_k|_1}s\left( \frac{u_k}{|u_k|_1},v_k\right),
    \end{equation*}
    which implies \begin{equation*}
       s\left( u_k,v_k\right)\geq  \frac{m}{|u_k|_1}.
    \end{equation*}
    Finally, since $u_k\to 0$ as $k\to \infty$, the desired conclusion follows immediately.
\end{proof}
Now, we continue with the proof of Lemma \ref{U open} stated above.
\begin{proof}[Proof of Lemma \ref{U open}]
To prove that $U$ is open in the
relative topology of $K_{1}\times D$, it suffices to show that the set
 $(K_{1}\times D) \setminus U$ is closed. To this aim, let $(u_k, v_k)\in (K_{1}\times D) \setminus U$ be any sequence convergent to some $(u,v)\in K_1\times D$. We need to show that $(u,v)\notin U$.

 First, we claim that \begin{equation}\label{element in Ub}
     \left(s\left(u_k,v_k\right)u_k,v_k\right)\in U_b.
 \end{equation}
Observe that, since $\{0\}\times D\subset U$ and $(u_k,v_k)\notin U$, we have $u_k\neq 0$. Now, using relation \eqref{prop_s_homogeneity}, we have \begin{equation*}
    s\left(s\left(u_k,v_k\right)u_k,v_k \right)=\frac{1}{s\left(u_k,v_k\right)}s\left(u_k,v_k\right)=1,
\end{equation*} 
 whence our claim is verified.  
 
Next, write
\begin{equation}\label{u_k v_k}
    (u_k, v_k) = \left( \frac{1}{s(u_k, v_k)} s(u_k, v_k) u_k, v_k \right).
\end{equation}
Since \( (u_k, v_k) \in (K_1 \times D) \setminus U \), by \eqref{element in Ub} it follows that 
\begin{equation}\label{inegalitate_s}
    \frac{1}{s(u_k, v_k)} \geq 1.
\end{equation}This inequality implies that the limit of the sequence \(u_k\) is nonzero, i.e., \(u \neq 0\). Indeed, if this were not the case, that is, if \(u = 0\), then by Proposition~\ref{lema_conv_zero} we would have
\[
s(u_k, v_k)\,\to\, \infty
 \,\, \text{ as } k \to \infty,
\]
which yields a contradiction with \eqref{inegalitate_s}.

The final step in our proof is to pass to the limit in \eqref{inegalitate_s}. To see why, note that $s(u_k,v_k)\to s(u,v)$. Also, by relation \eqref{inegalitate_s}, we have \begin{equation*}
    \frac{1}{s(u,v)}\geq 1.
\end{equation*}
Moreover, since $\left( s(u,v)u,v\right)\in U_b$, from the definition of $U$ we conclude that \begin{equation*}
   (u,v)= \left( \frac{1}{s(u,v)}s(u,v) u,v\right)\notin U,
\end{equation*}
as desired.

To show the second part, observe that we are concerned with the boundedness only with respect to the first component, since by definition, $D$ is a bounded set. 

Suppose now that there exists a sequence $(u_k,v_k)\in U_b$ such that $|u_k|_1\to \infty$. Then, using   \begin{equation*}
    1=s(u_k,v_k)=\frac{1}{|u_k|_1} s\left( \frac{u_k}{|u_k|_1},v_k\right)\leq \frac{1}{|u_k|_1} M,
\end{equation*}
by letting $k\to \infty$ we arrive at a contradiction, so the set $U_b$ is bounded. From this, the boundedness of $U$ follows immediately.
\end{proof}\begin{rem}
    The boundary of $U$ relative to $K_1\times D$ is  $U_b,$ i.e., $$\partial U=U_b.$$
\end{rem}
We are ready now to state the first main result of our paper.
\begin{thm}\label{thm principala}
Assume that conditions $\emph{(h1)-(h3)}$ are satisfied. Then, the operator \((T_1, T_2)\) admits a fixed point  \( (u,v)\in K_1 \times D\). If in addition the operator \((T_1, T_2)\) has no fixed point of the form \((0, \bar{v})\), with \(\bar{v} \in D\), then \((u,v)\in U_b\).
\end{thm} \begin{proof}

Let $\omega\in D$ and consider the homotopy
$H\colon \overline{U}\times [0,1]\to K_1\times D,$
\[
H\bigl((u,v),\,t\bigr) = \Bigl(\,t\,T_{1}(u,v),\; t\,T_{2}(u,v)\;+\;\bigl(1 - t\bigr)\,\omega\Bigr).
\] Clearly, $H$ is a compact mapping.
We distinguish two possible cases: \begin{itemize}
    \item[(a)] The homotopy $H$ has no fixed points on $\partial U$, i.e.,
    \begin{equation*}
        0\notin \bigl(I - H\bigr)\bigl(\,\partial U\times [0,1]\bigr),
    \end{equation*}
    or
    \item[(b)] There exists $(u,v)\in \partial U$ and $t\in [0,1]$ such that \begin{equation}\label{fixed point H}
       (u,v)=H((u,v),t).
    \end{equation}
 \end{itemize}
In case (a), we show that \((T_1, T_2)\) has only fixed points of the form \((0, v)\) with \(v \in D\). To prove the existence of such a fixed point, we use the homotopy invariance of the fixed point index, which ensures that
 \begin{equation*}
     \mathrm{ind}_{K_1\times D}((T_1,T_2),U)=\mathrm{ind}_{K_1\times D}(H(\cdot,1),U)=\mathrm{ind}_{K_1\times D}(H(\cdot,0),U).
 \end{equation*} Since $H(\cdot,0)=(0,\omega)\in U$, the normalization property of the fixed point index yields \begin{equation*}
   \mathrm{ind}_{K_1\times D}((T_1,T_2),U)  =\mathrm{ind}_{K_1\times D}((0,\omega),U)=1.
 \end{equation*} Therefore, the mapping $(T_1,T_2)$ has a fixed point in $U$, i.e., there exists  $(u,v)\in U$ such that\begin{equation*}
     u=T_1(u,v) \quad \text{ and }\quad v=T_2(u,v).
 \end{equation*} 
Assume that \(u \neq 0\). Then, using property \eqref{proprietate F}, we observe that
\[
\mathcal{F}\bigl(T_1(u,v),\, u \bigr) = \mathcal{F}\bigl(u,\, u\bigr) = 0.
\]
This  implies that $s(u,v)=1$, so $(u,v)\in U_b=\partial U$. However, this  leads to a contradiction since  $(u,v)\in U$ and $U$ is open. Therefore, $u=0$, so $(0,v)$ is a fixed point for $(T_1,T_2)$.

In the case (b), relation~\eqref{fixed point H} is equivalent to
\begin{equation}\label{fixed point H 2}
    u = t\, T_1(u,v) \quad \text{and} \quad v = t\, T_2(u,v) + (1 - t)\,\omega.
\end{equation}
If \(t = 0\), then \(u = 0\), which contradicts the assumption that \((u,v)\in \partial U\), since \((0,v)\in U\) and \(U\) is open.
If \(t \in (0,1]\), then from~\eqref{fixed point H 2} we have
\begin{equation}\label{T_1 eigenvalue}
    T_1(u,v) = \frac{1}{t}\,u.
\end{equation}
Given that \((u,v)\in U_b = \partial U\), and using~\eqref{T_1 eigenvalue}, by the definition of the set \(U_b\), it follows that
\[
    0 = \mathcal{F}\bigl( T_1(u,v),\,u \bigr) = \mathcal{F}\bigl( \tfrac{1}{t}\,u,\,u \bigr).
\]
Note that property~\eqref{proprietate F} implies  \(\tfrac{1}{t} = 1\), which holds only if \(t = 1\).
In this case, relation~\eqref{fixed point H 2} shows that \((u,v)\) is a fixed point of the operator \(\bigl(T_1,\,T_2\bigr)\).
\end{proof}
\subsection{Nehari type fixed point theorem}

Consider now a single fixed point equation
\begin{equation}
    T(u) = u, \label{epf}
\end{equation}
in a cone \( K \) of a Banach space \( (X, |\cdot|) \). Let us take \( X_1 := X \), \( K_1 := K \), and let \( D \) be a singleton, that is, \( D = \{v_0\} \) for some fixed \( v_0 \in X \). 
By identifying
\[
T_1(u, v_0) := T(u), \qquad T_2(u, v_0) := v_0,
\]
we recover the fixed point problem \((T_1(u, v), T_2(u, v)) = (u, v)\) in the special case where the second component remains constant. This yields the following fixed point principle.

Assume that $T:K\rightarrow K$ is completely continuous, $\mathcal{F}%
:K\times K\rightarrow 
\mathbb{R}
$ is any continuous functional satisfying (\ref{proprietate F}), and consider the following conditions:
\begin{description}
\item[(c1)]  For every $u\in K\setminus \{0\}$, there exists a unique number $s(u)>0$ such that 
\begin{equation*}
\mathcal{F}\Bigl(T\bigl(s(u)\,u\bigr),\,s(u)\,u\Bigr)=0.
\end{equation*}

\item[(c2)] Denoting $S_{K}^{1}:=\left\{ u\in K:\ \left\vert u\right\vert
=1\right\} ,$ there exists $0<m<M<\infty$ such that  
\begin{equation*}
\inf_{S_{K}^{1}}s(\cdot )>m
\end{equation*}
and
\begin{equation*}
\sup_{S_{K}^{1}}s\left( \cdot \right) <M .
\end{equation*}
\end{description}

\begin{cor}
Assume conditions \emph{(c1)-(c2)} hold. Then, the operator $T$ has a fixed
point $u\in $ $K$. In addition, if $\, 0$ is not a solution of (\ref{epf}), then 
$s\left( u\right) =1.$
\end{cor}
Remaining in the setting of a single equation, we now let \( X := H \) be a Hilbert space endowed with the inner product \( (\cdot, \cdot)_H \), and identified with its dual. Consider a $C^1$ functional
\[
E \colon H \to \mathbb{R},
\]
whose Fréchet derivative is denoted by \( E' \colon H \to H \). In this framework,  define the operator
\begin{equation}\label{T particular}
    T(u) = u - E'(u),
\end{equation}
and assume that \(T\) is completely continuous and satisfies the invariance condition
\[
T(K) \subset K.
\]
Further, considering the functional \(\mathcal{F} \colon K \times K \to \mathbb{R}\) given by
\begin{equation}\label{F particular}
    \mathcal{F}(\tilde{u}, u) := (u - \tilde{u}, u)_H,
\end{equation}
 we recover the classical Nehari manifold \(\mathcal{N}_K\) defined by \begin{equation}\label{N_k}
    \mathcal{N}_K=\left\{u\in K\setminus \{0\}\, : \, \left( E'(u),u\right)_H=0\right\}.
\end{equation}
Then, for any \(u \in K\), we have
\[
\mathcal{F}(T(u), u) = (u - T(u), u)_1 = (E'(u), u)_H.
\]
Therefore, based on Remark \ref{remarca_echivalenta_cu_nehari}, the set \(U_b\) becomes 
\[
U_b = \left\{ u \in K \setminus \{0\} \;:\; (E'(u), u)_H = 0 \right\},
\]
i.e., $U_b=\mathcal{N}_K$.

\begin{rem}
Returning to the general case of a Banach space \((X, |\cdot|_X)\) endowed with the cone $K$, and inspired by the construction of the Nehari manifold, one may consider the following example of a functional \(\mathcal{F}\). Assume that \(X\) is continuously embedded into a Hilbert space \((H, (\cdot, \cdot)_H)\). Then, \(\mathcal{F}\) can be defined by
\[
\mathcal{F}(\tilde{u}, u) := (u - \tilde{u}, u)_H \quad (u, \tilde{u}\in K).
\]


\end{rem}


\subsection{Nehari-Schaefer fixed point theorem}

The previous result was established under the invariance condition on \(T_2\),
namely,
\[
T_2\bigl(K_1 \times D\bigr) \,\subset\, D,
\]
where \(D\) is a given nonempty, bounded, closed, and convex set, as in Schauder's fixed point theorem. The next result does not require such an invariance condition.

Under the above notations, assume that, instead of \(D\), we consider a closed ball \(B_R\subset X_2\), centered at the origin and of radius \(R>0\). Assume further that the operators
\[
T_1 \colon K_1 \times X_2 \to K_1,
\quad
T_2 \colon K_1 \times B_R \to X_2,
\]
are completely continuous, and that the following conditions are satisfied:
\begin{description}
\item[(a1)] The set $T_{2}\left( K_{1}\times B_{R}\right) $ is bounded in $%
X_{2}.$

\item[(a2)] For every $(u,v)\in \left( K_{1}\setminus \{0\}\right) \times
B_{R}$, there exists a unique number $s(u,v)>0$ such that 
\begin{equation*}
\mathcal{F}\Bigl(T_{1}\bigl(s(u,v)\,u,\,v\bigr),\,s(u,v)\,u\Bigr)=0.
\end{equation*}

\item[(a3)] 

There exists $0<m<M<\infty$ such that 
\[
\inf_{S_{K_1}^{1}\times D} s(\cdot, \cdot) > m,
\]
and \begin{equation*}
\sup_{S_{K_1}^{1}\times D}s\left( \cdot ,\cdot \right) <M.
\end{equation*}
\end{description}

\begin{thm}
\label{t2} Assume conditions \emph{(a1)-(a3)} hold. In addition, assume that %
\begin{equation*}
v\neq \lambda T_{2}\left( u,v\right) \ \ \ \text{for all }u\in K_1, v\in \partial
B_{R},\text{ and } \lambda \in \left( 0,1\right) .
\end{equation*}%
Then, the operator $(T_{1},T_{2})$ has a fixed point in $K_{1}\times B_{R}.$
\end{thm}

\begin{proof}
We reduce the problem to that one from Section 2.1. To this aim,
we let%
\begin{equation*}
D:=B_{\widetilde{R}},
\end{equation*}%
where $\widetilde{R}\geq R$ is such that $$T_{2}\left( K_{1}\times
B_{R}\right) \subset B_{\widetilde{R}}.$$ Such a ball exists based on
assumption (a1). 
Next, we extend the operators \(T_1\) and \(T_2\) from \(K_1 \times B_R\) to  
\(\,K_1 \times B_{\widetilde{R}}\,\) by defining
\[
\widetilde{T}_1 \colon K_1 \times B_{\widetilde{R}} \to K_1 \quad \text{and} \quad \widetilde{T}_2 \colon K_1 \times B_{\widetilde{R}} \to B_{\widetilde{R}},
\]
as follows
\[
\widetilde{T}_i(u, v) = 
\begin{cases}
T_i(u, v), & \text{if } 0 \leq |v|_2 \leq R, \\
T_i\left(u, \frac{R}{|v|_2} v \right), & \text{if } R < |v|_2 \leq \widetilde{R},
\end{cases}
\quad \text{for } i = 1,2.
\]
Letting   $\widetilde{T}:=\left( \widetilde{T}_{1},%
\widetilde{T}_{2}\right) ,$ we aim to apply
 Theorem \ref{thm principala}. 
 First, note that the operator $\widetilde{T}$ is completely continuous and invariant with respect to  $K_1 \times B_{\widetilde{R}}$. Next, we show that the mapping
 \begin{equation*}
     \widetilde{s}\colon \left( K_{1}\setminus
\{0\}\right) \times B_{\widetilde{R}}\to (0,\infty),\quad \widetilde{s}\left( u,v\right) =\left\{ 
\begin{array}{lll}
s\left( u,v\right)  & \text{if} & \left\vert v\right\vert _{2}\leq R \\ 
s\left( u,\frac{R}{\left\vert v\right\vert_2 }v\right)  & \text{if} & 
R<\left\vert v\right\vert _{2}\leq \widetilde{R},
\end{array}%
\right. 
 \end{equation*}
satisfies (h1). Let \( u \in K_1 \setminus \{0\} \). If \( v \in B_R \), the statement follows directly from assumption (a2). Now, consider \( v \in B_{\widetilde{R}} \setminus B_R \), and define \( w := \frac{R}{|v|_2} v \in B_R \). Then, we have
\begin{align*}
\mathcal{F}\left( \widetilde{T}_1\left( \tilde{s}(u, v)\, u, v \right),\, \tilde{s}(u, v)\, u \right)
&= \mathcal{F}\left( T_1\left( s(u, w)\, u, w \right),\, s(u, w)\, u \right) = 0,
\end{align*}
by assumption (a2) and the fact that \( w \in B_R \). Clearly, assumption (h2) and (h3) also are satisfied.
 Thus, Theorem \ref{thm principala} applies and guarantees the
existence of a fixed point $\left( u,v\right) $ of $\widetilde{T},$ with $%
u\in K_{1}$ and $v\in B_{\widetilde{R}}.$ It remains to show that in fact $%
v\in B_{R}.$ Assume, by contradiction, that \(v \in B_{\widetilde{R}} \setminus B_R\). Then, by the definition of \(\widetilde{T}_2\), we have  
\[
v=\widetilde{T}_2(u, v) = T_2\left(u, \frac{R}{|v|_2} v \right).
\]
Denote \( w := \frac{R}{|v|_2} v \) and \( \lambda := \frac{R}{|v|_2} \). Then,
\[
w \in \partial B_R, \quad \lambda \in (0,1), \quad \text{and} \quad \lambda\, T_2(u, w) = w,
\]
which contradicts assumption (a1). Hence, \(v \in B_R\), and therefore
\[
 (u, v)=\widetilde{T}(u, v) = T(u, v) ,
\]
which completes the proof.

\end{proof}


\section{Application}

We now illustrate the applicability of Theorem \ref{thm principala} by considering the system \begin{equation}\label{pb aplicatie}
    \begin{cases}
        u(t)=\int_0^1 k_1(t,\theta) f(u(\theta), v(\theta))ds:=T_1(u,v)(t)\\
        v(t)=\int_0^1 k_2(t,\theta) g(u(\theta),v(\theta))ds:=T_2(u,v)(t),
    \end{cases}
\end{equation}
where \(k_1,k_2\) are nonnegative continuous functions on \([0,1]^2\);   $f,g\in C\left(\mathbb{R}^2,\mathbb{R} \right)$ and $f(x,y)\geq 0$ for all \(x\geq 0\) and  $y\in \mathbb{R}$.  Related to the kernel $k_1$,  assume that 
\begin{description}
    \item[(H1)] There exists a continuous function $\Phi\colon [0,1]\to \mathbb{R}_+$, an interval $[a,b]\subset [0,1]$ and a constant $c_1>0$  such that \begin{align}
     \label{inegalitate K dreapta}  & k_1(t,\theta)\leq \Phi(\theta) \quad \text{ for all }\, \,  t,\theta\in [0,1],\\&
       c_1\Phi(\theta)\leq k_1(t,\theta)\quad \text{ for all }\,\, t\in [a,b]\text{ and }\theta\in [0,1].  \label{inegalitate K stanga}
    \end{align} 
\end{description}
In order to apply Theorem~\ref{thm principala}, take
\[
X_1 = X_2 = C\bigl([0,1],\,\mathbb{R}\bigr),
\]
endowed with the supremum norm
\[
|u|_\infty = \max_{t\in [0,1]} \,|u(t)|.
\]
By standard arguments (see, e.g., \cite{infante}), both operators $T_1, T_2$ are completely continuous from $C\bigl([0,1],\,\mathbb{R}\bigr)^2 \to C\bigl([0,1],\,\mathbb{R}\bigr).$ 
Furthermore, in $C\bigl([0,1],\,\mathbb{R}\bigr)$ we consider the cone
\[
K_1 := \bigl\{ u \in C\bigl([0,1],\,\mathbb{R}\bigr)\, : \,u\geq 0\,\, \text{ and }\, \,\min_{t\in [a,b]} u(t)\,\ge c_1\,|u|_\infty \bigr\},
\]
and the bounded, closed, and convex set \begin{equation*}
    D=B_R=\bigl\{u\in C\bigl([0,1],\,\mathbb{R}\bigr)\, : \, |u|_\infty\leq R \bigl\},
\end{equation*} where $R>0$ is a given number. 
  The functional $\,\mathcal{F}\colon K_1\times K_1\to \mathbb{R}\,$ is chosen to be
\[
\mathcal{F}(u,\tilde{u}) = \int_0^1 \bigl(u(t)-\tilde{u}(t)\bigr)\,u(t)\,dt.
\]
Additionally, define
\[
\alpha_1 := \max_{t\in [0,1]} \Phi(t), \, \,\, \, \alpha_2 := \int_a^b \int_a^b k_1(t,\theta)\,d\theta\,dt
\,\,\, \,
\text{and}\, \,
\,\,
\alpha_3=\max_{t\in [0,1]}\int_0^1 
        k_2(t,\theta)d\theta.
\]
Related to the functions $f,g$, we assume that the following conditions are satisfied.\begin{description} 
    \item[(H2)] There exists a positive constant $\,c_2<\frac{1}{\alpha_1}$ such that \begin{equation*}
        \lim_{x\searrow 0} \frac{f(x,y)}{x}\leq c_2 \quad \text{ for all }\, \,|y|\leq R.
    \end{equation*}
     \item[(H3)]
     There exists \, $0<c_3\leq \infty$ 
     \,such that 
     \begin{equation}\label{liminf}
          \lim_{x\to \infty} \frac{f(x,y)}{x}\geq c_3 \quad \text{ for all }\, \,|y|\leq R,
     \end{equation}
    and \begin{equation}\label{cond_const_c_3}
        c_3 >\frac{1}{ c_1^2 \alpha_2}.
    \end{equation}
    \item[(H4)] For each $y$ with $|y|\leq R$, the mapping \begin{equation*}
        x\mapsto \frac{f(x,y)}{x}
    \end{equation*}
    is strictly increasing on $(0,\infty)$. 
  \item[(H5)]
   For every  $x\geq0$ and any $y$ with $|y|\leq R$, one has\begin{equation*}
        g(x,y)\leq \alpha,
    \end{equation*}
    where 
    \begin{equation*}
        \alpha <\frac{1}{\alpha_3}.
    \end{equation*}
\end{description}
From $\text{(H5)}$, one clearly has 
\[
T_{2}\bigl(K_{1} \times D\bigr) \;\subset\; D\,\, \, \,\, \,(D=B_R).
\]
Also, for any \(u\in K_{1}\) and  \(v\in D\), using  \eqref{inegalitate K dreapta} and \eqref{inegalitate K stanga}, we have
\begin{align*}
    T_{1}(u,v)(t)
    &= \int_{0}^{1} k_1(t,\theta)\,f\bigl(u(\theta),v(\theta)\bigr)d\theta \\
    &\geq c_{1}\int_{0}^{1} \Phi(\theta)\,f\bigl(u(\theta),v(\theta)\bigr)d\theta \\
    &\geq c_{1}\int_{0}^{1} k_1(t',\theta)\,f\bigl(u(\theta),v(\theta)\bigr)d\theta \\
    &= c_{1}\,T_{1}(u,v)(t')
\end{align*}
for all \(t \in [a,b]\) and $t'\in [0,1]$. Consequently, \begin{equation*}
    T_{1}\bigl(K_{1} \times B_{R}\bigr) \;\subset\; K_1.
\end{equation*}
Note that for any  $u\in K_1\setminus\{0\}$, $v\in D$ and $\sigma>0$, one may write
 \begin{align}\label{forma F}
    & \mathcal{F}\left( T_1(\sigma u,v),\sigma u\right)\\&=\int_0^1 T_1(\sigma u(t),v(t))\sigma u(t)dt-\sigma ^2\int_0^1 u^2 (t)dt
     \\&
     =\sigma ^2\left( \int_0^1 \left(\int_0^1 \frac{1}{\sigma }k_1(t,\theta) f(\sigma u(\theta),v(\theta)d\theta \right)u(t)dt-\int_0^1 u(t)^2dt\right). \nonumber
\end{align}
In the sequel, without further mention, we will use the   fact that whenever $v\in D$, one has
\[
|v(t)| \le R \,\, \text{ for all } t\in [0,1].
\]

\smallskip

\textbf{Check of condition (h2)}. 
 Let $u\in K_1$ with $|u|_\infty=1$, and $v\in D$. 
From assumption $\text{(H2)}$, there exists $\delta_0>0$ such that \begin{equation}\label{inegalitate f spre zero}
   f(x,y) \leq c_2x \,\, \text{ for all } \,\, 0\leq x<\delta_0 \text{ and} \,\,|y|\leq R.
\end{equation} Thus, for $0\leq \sigma<\delta_0$, using $\text{(h1)}$, \eqref{inegalitate f spre zero} and the Holder's inequality,  we estimate \begin{align*}
   \int_0^1\left(\int_0^1  \frac{1}{\sigma }k_1(t,\theta) f(\sigma u(\theta),v(\theta)d\theta \right)u(t)dt& \leq  \int_0^1 \Phi(\theta) u(\theta)^2 d\theta\int_0^1 u(t)dt\\&
   \leq c_2 \max_{[0,1]}\Phi(\cdot)\left( \int_0^1 u(t)dt\right)^2\\&
   \leq  \frac{c_2}{\alpha_1}\int_0^1 u^2(t)dt.
\end{align*}
Since $c_2<\frac{1}{\alpha_1}$, we see that
\begin{align*}
    \mathcal{F}\bigl( T_1(\sigma u,v),\,\sigma \,u \bigr)\leq \sigma ^2 \left( c_2 \frac{1}{\alpha_1}\Phi(\cdot)-1\right)\int_0^1 u^2(t)dt<0,
\end{align*}
for all \(0\leq \sigma <\delta_0.\) Consequently, if \(\sigma \) satisfies
\[
\mathcal{F}\bigl( T_1(\sigma u,v),\,\sigma \,u \bigr)=0,
\]
then necessarily \(\sigma \ge \delta_0.\) Since \(\delta_0\) is independent of the choice of \(u\) and \(v\), condition~$\text{(h2)}$ is satisfied.

To verify that $\sigma<M$ for some $M>0$ whenever $|u|_\infty=1$ and $|v|_\infty\leq R$, note that 
using condition \eqref{liminf} from $\text{(H3)}$, there exist constants $\sigma_0>0$ and $   c_4>\frac{1}{c_1^2 \alpha_2}$ such that 
 \begin{equation}\label{marginire inferioara f}
     f(\sigma x,y)\geq \sigma  c_4 x\,\, \text{ for all }\sigma \geq \sigma _0,\, x\geq c_1 \,\text{ and  }\,|y|\leq R.
 \end{equation} 
Let $t\in [0,1]$  and $\sigma \geq \sigma _0$. Using $\text{(H1)}$ and \eqref{marginire inferioara f}, we estimate
\begin{align*}
    \int_0^1 \frac{1}{\sigma }k_1(t,\theta) f(\sigma u(\theta),v(\theta)d\theta
    & 
    \geq \int _a^b \frac{1}{\sigma }k_1(t,\theta)f(\sigma u(\theta),v(\theta)d\theta\\&
      \geq  c_4 \int _a^b k_1(t,\theta) u(\theta)d\theta\\&
      \geq  c_1c_4\int _a^b k_1(t,\theta)d\theta.
\end{align*}
Therefore, we further have  \begin{align*}
    \int_0^1 \left( \int_0^1 \frac{1}{\sigma }k_1(t,\theta) f(\sigma u(\theta),v(\theta)d\theta\right) u(t)dt &
    \geq  c_1c_4 \int_0^1 \left(\int _a^b k_1(t,\theta)d\theta \right)u(t)dt\\&
    \geq  c_1c_4 \int_a^b \left(\int _a^b k_1(t,\theta)d\theta \right)u(t)dt
    \\& \geq c_1^2 c_4\int_a^b \int _a^b k_1(t,\theta)d\theta dt\\&
 =  c_1^2 c_4\alpha_2.
\end{align*}
Applying the above estimate in \eqref{forma F}, together with the obvious inequality \begin{equation*}
    \int_0^1 u(t)^2 dt\leq |u|_\infty^2=1,
\end{equation*} and relation \eqref{cond_const_c_3}, we obtain
\begin{align*}
      \mathcal{F}\left( T_1(\sigma u,v),\sigma u\right)&\geq \sigma ^2 \left( c_1^2 c_4 \alpha-\int_0^1 u(t)^2dt\right)\\&
      \geq \sigma ^2 \left( c_1^2 c_4  \alpha_2 -1\right)\\& >0.
\end{align*}
Consequently, if $   \mathcal{F}\left( T_1(\sigma u,v),\sigma u\right)=0$ for some $|u|_\infty$ and $|v|_\infty\leq R$, then we necessarily  have $\sigma \leq \sigma _0$. Thus, condition $\text{(h2)}$ is verified.

\smallskip

\textbf{Check of condition (h1).} 
Let \(u \in K_1 \setminus \{0\}\) and $v\in D$. From~$\text{(H4)}$, we see that the mapping
\[
\sigma  \mapsto \int_0^1 \left( \int_0^1 \frac{1}{\sigma }\,k_1(t,\theta)\,f\bigl(\sigma \,u(\theta),\,v(\theta)\bigr)\,d\theta \right)\,u(t)\,dt
-
\int_0^1 u^2(t)\,dt
\]
is strictly increasing. Therefore, corroborating this with (h2), we conclude that there exists a unique \(\sigma^\ast>0\) such that
\[
\mathcal{F}\bigl( T_1(\sigma^\ast\,u,\,v),\,\sigma^\ast\,u \bigr) = 0.
\]
Thus, defining $s(u,v):=\sigma^\ast$, condition $\text{(h1)}$ follows.

\begin{rem}[Typical examples of function \(f\)]

Let \( f_1 \colon \mathbb{R} \to \mathbb{R} \) be a continuous and positive function. Then the function
\[
f(x, y) = |x|^p\, f_1(y),
\]
with \( p > 1 \), satisfies conditions (H2)–(H4).  
The function \( g \) can be any continuous function bounded by a  constant as specified in (H5).

\end{rem}

\begin{rem}
Integral systems like \eqref{pb aplicatie} often arise from the reformulation of bi-local boundary value problems, where the kernels \(k_1\) and \(k_2\) correspond to Green’s functions \cite{cabada2014greens}. For this class of problems, properties \eqref{inegalitate K dreapta} and \eqref{inegalitate K stanga} are properties of the Green’s function (see, e.g., \cite{erbe1994positive, precup_rodriguez2023}), where condition \eqref{inegalitate K stanga} is a Harnack-type inequality \cite{precup_rodriguez2023}.

\end{rem}
\phantom{ }
\newline
\textbf{Acknowledgements}
The authors thank the reviewers for taking the time to read the manuscript and for their appreciations. \phantom{ }
\newline
\textbf{Author contributions}
 Both authors have contributed equally to the preparation
of this manuscript. Both authors read and approved the final manuscript.
\phantom{ }
\phantom{ }
\newline
\textbf{Funding}
Authors did not receive any funding for this work.
\phantom{ }

\phantom{ }
\newline
\textbf{Data Availability }
No datasets were generated or analyzed during the current study.
\phantom{ }
\newline
\textbf{Declarations}
\phantom{ }

\phantom{ }
\newline
\textbf{Conflict of interest} The authors have no relevant financial or non-financial
interests to disclose.

\end{document}